\documentclass[letter, 11pt]{article}
\usepackage[english]{babel}
\usepackage[utf8]{inputenc}

\usepackage[margin = 1.2 in, top = 1 in, bottom = 1.2 in]{geometry}
\makeatletter
\g@addto@macro\normalsize{%
  \setlength\abovedisplayskip{7pt}
  \setlength\belowdisplayskip{7pt}
  \setlength\abovedisplayshortskip{7pt}
  \setlength\belowdisplayshortskip{7pt}
}
\makeatother

\usepackage{graphicx}
\usepackage{cancel}

\usepackage{enumerate}
\usepackage{scalerel,stackengine}
\usepackage{graphics}
\usepackage{fancyhdr}
\usepackage{cancel}
\usepackage{enumerate}
\usepackage{amssymb}
\usepackage{amsmath}
\usepackage{hyperref}
\usepackage{mdwlist}
\usepackage{etoolbox}
\usepackage{latexsym}
\usepackage{amsthm}
\usepackage{multicol}
\usepackage{makeidx}
\usepackage{bm}
\usepackage{dsfont}
\usepackage{graphicx}
\usepackage{wrapfig}
\usepackage{multicol}
\usepackage{makeidx}
\usepackage{bm}
\usepackage{dsfont}
\usepackage{mdframed,color}
\usepackage{amsmath,amssymb}
\usepackage[dvipsnames]{xcolor}
\usepackage{mathtools}

\usepackage[capitalize]{cleveref}
\usepackage{bbm, dsfont}
\usepackage{tikz}
\usetikzlibrary{matrix}
\usetikzlibrary{positioning}

\usepackage{titlesec}
\titleformat{\section}
  {\large\center\bfseries}
  {\thesection.}{.7em}{}
\titlespacing*{\section}{0pt}{3.5ex plus 0ex minus 0ex}{1.5ex plus 0ex}
\titleformat{\subsection}
  {\center\bfseries}
  {\thesubsection.}{.7em}{}
\titlespacing*{\subsection}{0pt}{3.5ex plus 0ex minus 0ex}{1.5ex plus 0ex}
\titleformat{\subsubsection}
  {\center\bfseries}
  {\thesubsubsection.}{.7em}{}
\titlespacing*{\subsubsection}{0pt}{3.5ex plus 0ex minus 0ex}{1.5ex plus 0ex}

\addto\captionsenglish{}

\usepackage{titling}
\newtheorem{theorem}{Theorem}[section]	
\newtheorem{corollary}[theorem]{Corollary}

\newtheorem{lemma}[theorem]{Lemma}
\newtheorem{proposition}[theorem]{Proposition}

\newtheoremstyle{definition}{2mm}{2mm}{}{}{\bfseries}{.}{.5em}{}
\theoremstyle{definition}
\newtheorem{definition}[theorem]{Definition}
\newtheorem{remark}[theorem]{Remark}
\newtheorem*{remark*}{Remark}
\newtheorem{example}[theorem]{Example}
\theoremstyle{plain}
\newtheorem*{namedthm}{\namedthmname} 



\newcommand{\N}{\mathbb{N}}
\newcommand{\Z}{\mathbb{Z}}

\newcommand{\R}{\mathbb{R}}

\newcommand{\bS}{\mathbb{S}}

\newcommand{\T}{\mathbb{T}}
\newcommand{\norm}[1]{\left\lVert#1\right\rVert}
\newcommand{\ind}[1]{\mathbbm{1}_{#1}}

\providecommand{\norm}[1]{\lVert #1\rVert}

\usepackage[normalem]{ulem}

\newcommand{\E}{\operatorname{\mathbb{E}}}

\begin{document}
\author{By~~{\scshape Felipe~Hernández~}}
\date{\small \today}
\title{{\bfseries Structure and Spectrum of Nonergodic Nilsystems}}
\maketitle

\begin{abstract}
We study Host–Kra factors and the spectral type of nilsystems, without assuming ergodicity.
In the ergodic case, it is known that the spectral type splits into a discrete component and a Lebesgue component of infinite multiplicity.
Our main result extends this decomposition to nonergodic nilsystems.
\end{abstract}
\section{Introduction} 
 Let $G$ be a $k$-step nilpotent Lie group and $\Gamma$ a discrete cocompact subgroup of $G$. The compact manifold $X=G/\Gamma$ is called a \textit{$k$-step nilmanifold}. For an element $\tau \in G$ we denote by $T_\tau$ the transformation on $X$ given by left translation: $T_\tau(g\Gamma)=\tau g\Gamma$ for each $g\in G$. If it is clear from the context, we will denote such a transformation just by $T$. Each nilmanifold $X=G/\Gamma$ admits a unique \textit{Haar measure}, that is, a left-translation invariant Borel probability measure $\mu_X$. The pair $(X,T)$ is a topological dynamical system, while the triple $(X,\mu_X,T)$ is a measure-preserving system. When we refer to a \textit{nilsystem}, we regard both topological and measurable subjacent structures, and it will be clear from the context which one is meant. This is a common practice (see for example \cite[Chapter 11]{Host_Kra_nilpotent_structures_ergodic_theory:2018}). We will assume without loss of generality throughout this article, that for any nilsystem $(X=G/\Gamma,\mu_X,T_\tau)$ the group $G$ is generated by the connected component of the identity $G^\circ$ and $\tau$ (see \cref{section-2.1} for more details).

Nilsystems have been widely studied in the history of dynamical systems. Basic properties of nilsystems and their spectra were first studied by Auslander, Green, and Hahn \cite{Auslander_Green_Hahn63}, and criteria for ergodicity and minimality, as well as convergence of ergodic averages in nilsystems, can be found in the works of Parry \cite{Parry70}, Lesigne \cite{Lesigne_1991}, and Leibman \cite{Leibman05a,Leibman_pointwise_average_polynomial_nil_Zd:2005}, among many others. (We refer to \cite{Host_Kra_nilpotent_structures_ergodic_theory:2018} for a more complete exposition.) Moreover, in recent years the study of nilsystems in ergodic theory has gained even greater importance due to their connection to the structure theory for ergodic averages \cite{Host_Kra_nonconventional_averages_nilmanifolds:2005}. In particular, the discrete-spectrum factor and the Host--Kra $\mathcal{Z}_k$-factor of a system are topics of paramount importance. For background and notation, we refer the reader to \cref{sec-2.15}. Our first theorem gives an explicit algebraic description of the $\mathcal{Z}_k$-factors of nilsystems. We remark that this was well-known for ergodic nilsystems (see \cite[Cf. Chapters 11 and 13]{Host_Kra_nilpotent_structures_ergodic_theory:2018}). For nonergodic nilsystems, the only previously understood case is when $k=0$ and $X$ is connected, which is implicit in the work of Leibman (see \cite[Theorem 2.2]{Leibman03}).
\begin{theorem}\label{Theo-A}
Let $(X=G/\Gamma,\mu,T)$ be a nilsystem. There is a normal subgroup $H\subseteq  G$ such that $H$ is rational (i.e. $H\Gamma$ is closed) and the system $(G/H_{k+1}\Gamma, \mu_{G/H_{k+1}\Gamma},T) $ is the $\mathcal{Z}_k$-factor of $X$ for each $k\in \N_0$, where $H_{k+1}$ denotes the $(k+1)$-th term of the lower central series of $H$. In particular, the group $H$ satisfies $\overline{\{ T^nx: n\in \N\} }= Hx $ for $\mu$-a.e. $x\in X$. 
\end{theorem}
We remark that \cref{Theo-A} relies on the known results in the ergodic case and uses a special case
\cite[Theorem 2.2]{Leibman03} in its proof.

For instance, if $(X=G/\Gamma,\mu,T)$ is an ergodic nilsystem, then $H=G$ and the $\mathcal{Z}_k$-factor is $(G/G_{k+1}\Gamma,\ \mu_{G/G_{k+1}\Gamma},\ T).$
Although in the ergodic setting the $\mathcal{Z}_1$-factor coincides with the discrete-spectrum factor,
this does not necessarily hold in the nonergodic case (see \cref{example-2.1}). Thus, to identify the
discrete-spectrum factor for a nonergodic nilsystem $(X=G/\Gamma,\mu,T)$, we need to study its spectrum. The necessary background from spectral theory is recalled in \cref{sec-2.2}.

The spectral analysis of ergodic transformations on nilmanifolds begins with the work of L.\ Green
\cite{Auslander_Green_Hahn63}, who showed that in an ergodic nilsystem $(X=G/\Gamma,\mu_X, T)$ where $G$ is connected and simply connected, the spectrum splits into a discrete component and an infinite Lebesgue component. This decomposition was established by Host, Kra, and Maass in the case of $2$-step nilsystems in \cite{HKM}; and by Ackelsberg, Richter, and Shalom without any further hypotheses beyond ergodicity in \cite[Theorem 1.4]{MR4797105}. Our
next theorem shows that ergodicity can be dropped as well. For a subgroup $H\leq G$, we write
\[
\mathcal{J}(H,\Gamma):=\bigl(\overline{\langle H,\Gamma\rangle}\bigr)^{\circ},
\]
that is, the identity component of the closure of the subgroup generated by $H$ and $\Gamma$. This
notation was introduced by Auslander \cite{Auslander-63} and later used by Ghorbel and Loksaier
\cite{GK-19} in their study of the $\Gamma$-rational closure of connected subgroups of $G$.

\begin{theorem}\label{Theo-B}
Let $(X=G/\Gamma,\mu,T_{\tau})$ be a nilsystem. Then the Koopman representation of $T_{\tau}$ on $L^2(X,\mu)$ decomposes as
  $$ L^2(X)=L^2(\mathcal{J}([\tau,G],\Gamma)\backslash X)\oplus L^2(\mathcal{J}([\tau,G],\Gamma)\backslash X)^\perp,$$
  where $T_{\tau}$ exhibits discrete spectrum on $L^2(\mathcal{J}([\tau,G],\Gamma)\backslash X)$ and infinite Lebesgue spectrum on $L^2(\mathcal{J}([\tau,G],\Gamma)\backslash X)^\perp$ if nontrivial. In particular, $G/\mathcal{J}([\tau,G],\Gamma)\Gamma$ is the discrete-spectrum factor of $X$.
\end{theorem}
If $(X=G/\Gamma,\mu,T_\tau)$ is ergodic, then $\mathcal{J}([\tau,G],\Gamma)=[G,G]$ (see \cref{theo-B-erg}). This recovers the fact that the Kronecker factor of $(X=G/\Gamma,\mu,T)$ is $G/G_2\Gamma$, as well as the
result of Ackelsberg, Richter, and Shalom. However, we emphasize that our approach relies on these results. In the general (nonergodic) case, it is unclear to the author whether
$\mathcal{J}([\tau,G],\Gamma)=[H,G]$, where $H$ is the group given by \cref{Theo-A}. 

We obtain the following immediate corollary from \cref{Theo-B}.
\begin{corollary}\label{imm-coro}
    Let $(X=G/\Gamma,\mu,T)$ be a nilsystem. Let $f\in L^\infty(X)$ and suppose that $f$ is orthogonal to the subspace spanned by all eigenfunctions of $T$.  Then
    $$\lim_{n\to \infty} \int_X \overline{f}\cdot T^n f d\mu =0. $$
\end{corollary}
\cref{imm-coro} was known in the ergodic case previously by several sources \cite{Griesmer2,FranKuca,MR4797105}. In particular, its proof is identical to the one of  \cite[Theorem 1.7]{MR4797105} for the ergodic case, so we will omit it. 
\subsection*{Acknowledgements}
The author is grateful to Florian Richter for his guidance during the preparation of this article. We also thank Ethan Ackelsberg for fruitful discussion around the topic. Finally, we thank Bryna Kra for numerous suggestions on an early draft of this article and Or Shalom for pointing out \cref{imm-coro}.

\section{Preliminaries}
\subsection{Background from nilsystems}\label{section-2.1}
For a Lie group $G$ and $a,b\in G$, the \textit{commutator} of $a$ and $b$ is $[a,b]=aba^{-1}b^{-1}$. For subsets $A,B\subseteq G$, we write $[A,B]$ for the group generated by $\{[a,b]: a\in A, b\in B\}$. Whenever $A$ has only one element $a$, we will make the abuse of notation $[a,B]:=[A,B]$. For a nilpotent Lie group $G$, we denote its commutator subgroups as $G_1=G$ and $G_l=\left[G_{l-1}, G\right]$ for $l \geq 2$. Let $\Gamma$ be a discrete cocompact subgroup of $G$. A subgroup $H\subseteq G$ is called $\Gamma$-\textit{rational} (or simply \textit{rational} if $\Gamma$ is understood from the context) if $H\Gamma$ is closed. Notice that if $H$ is rational, then $H$ is closed (see \cite[Lemma 14 Chapter 10]{Host_Kra_nilpotent_structures_ergodic_theory:2018}). We denote by $G^{\circ}$ the connected component of the identity $e_G$ in $G$, and we denote by $m_{G^{\circ}}$ a Haar measure therein. Then, $G^{\circ}$ is an open normal subgroup of $G$ (see \cite[Section 4.1]{Host_Kra_Maass_variations_top_recurrence:2016}). By \cite[Chapter 10, Lemma 7 and Theorem 13]{Host_Kra_nilpotent_structures_ergodic_theory:2018}, we will assume without loss of generality that $G$ is simply connected (meaning that $G^{\circ}$ is simply connected) and that if $\tau\in G$ is a fixed element such that $G=\langle G^\circ,\tau\rangle$, then $\tau$ spans a free abelian group with trivial intersection with $G^\circ$.

For nilsystems, topological properties like transitivity, minimality, and unique ergodicity are equivalent to the measurable property of ergodicity (see {\cite{Parry_ergodic_affine_nil:1969,Lesigne_1991,Leibman05a,Host_Kra_nilpotent_structures_ergodic_theory:2018}}). Also, if $(X=G/\Gamma,\mu_X,T_\tau)$ is a nilsystem, we will make the standard assumption that $G$ is generated by $G^\circ$ and $\tau$, as $X$ can be decomposed into finitely many isomorphic spaces with this property (see \cite[Chapter 11, Section 1.2]{Host_Kra_nilpotent_structures_ergodic_theory:2018}). This, in particular, implies that the group $[\tau,G]=[\tau,G^\circ]$ is normal. Indeed, as $G^\circ$ is normal in $G$ and $G$ is generated by $G^\circ$ and $\tau$, it is enough to see that $[\tau,G^\circ]$ is normalized by $G^\circ$. Let $g,h\in G^\circ$. We have that $h[\tau,g]h^{-1}=[\tau,h]^{-1}[\tau,hg]\in [\tau,G^\circ], $ concluding that $[\tau,G^\circ]$ is normal in $G$. By \cite[Theorem 3.5]{GK-19}, this implies that $\mathcal{J}([\tau,G],\Gamma)$ is the smallest rational connected normal subgroup of $G^\circ$ that contains $[\tau,G]$. Also, by \cite[Proposition 3.4]{GK-19} we get that $\mathcal{J}([\tau,G],\Gamma)= (\overline{[\tau,G]\Gamma})^{\circ}$ and $\mathcal{J}([\tau,G],\Gamma)\Gamma=\overline{[\tau,G]\Gamma}$. As $\mathcal{J}([\tau,G],\Gamma)$ is trivially normalized by $\tau$, the group $\mathcal{J}([\tau,G],\Gamma)$ is normal in $G$.

A subnilmanifold $Y$ of $X$ is a set of the form $Y=Hx$ where $x\in X$ and $H\subseteq G$ is a rational subgroup of $G$. In this case, the subnilmanifold $Y$ is isomorphic to $H/(H\cap (g\Gamma g^{-1}))$ where $g\in G$ is such that $x=g\Gamma$. Consider now a nilsystem $(X=G/\Gamma,\mu_X,T)$. For a point $x\in X$, we denote its closed orbit under the transformation $T$ by 
$\overline{\mathrm{Orb}}_T(x):=\overline{ \{T^n x : n\in \N \}}. $
It is a known fact that such an orbit is a subnilmanifold of $X$ (see \cite[Chapter 11]{Host_Kra_nilpotent_structures_ergodic_theory:2018} for example). A subnilmanifold $Y$ of $X=G/\Gamma$ is called \textit{normal} if $Y=Hx$ where $x\in X$ and $H$ is a normal subgroup of $G^{\circ}$.

\subsection{Background from structure theory}\label{sec-2.15}

Let $(X,\mu,T)$ be a measure-preserving system. We will define the uniformity seminorms, which were introduced in the ergodic case in \cite{Host_Kra_nonconventional_averages_nilmanifolds:2005} and then in the nonergodic case in \cite{Chu_Frantzikinakis_Host11}. Let $f\in L^\infty(X)$. For $s=0$, we define
$$\norm{f}_{U^0(X,T,\mu)}:=\int_X f d\mu.$$
For $s\geq 0$, we define 
\begin{equation}\label{eq-US}
   \norm{f}_{U^{s+1}(X,\mu,T)}:=\lim_{H\to \infty} \left( \frac{1}{H}\sum_{h=0}^{H-1} \norm{T^h f \cdot \overline{f}}_{U^{s}(X,\mu,T)}^{2^s} \right)^{1/2^{s+1}}. 
\end{equation}
In \cite{Host_Kra_nonconventional_averages_nilmanifolds:2005} it is proved that the limit in \cref{eq-US} always exists and that for $s\geq 1$, $\norm{\cdot}_{U^s(X,\mu,T)}$ defines a seminorm on $L^\infty(\mu)$. The \textit{$\mathcal{Z}_s$-factor of} $(X,\mu,T)$ is the factor $(\mathcal{Z}_s,m_{\mathcal{Z}_s},T)$ such that for any $f\in L^\infty(X)$ we have 
$$\norm{f}_{U^{s+1}(X,\mu,T)}=0 \iff \E(f |\mathcal{Z}_s)=0. $$
If in addition $(X,\mu,T)$ is ergodic, then the main result in \cite{Host_Kra_nonconventional_averages_nilmanifolds:2005} states that $(\mathcal{Z}_s,m_{\mathcal{Z}_s},T)$ is the maximal factor of $(X,\mu,T)$ that is isomorphic to an inverse limit of $s$-step nilsystems. We recall that for a measure-preserving system $(X,\mu,T)$, the \textit{ergodic decomposition} of $\mu$ is the desintegration $(\mu_x)_{x\in X}$ of $\mu$ with respect the sigma-algebra of $T$-invariant functions $\mathcal{I}(T)$. Let $(X,\mu,T)$ be a measure-preserving system and $k\in \N_0$ be a nonnegative integer. Let $(\mu_x)_{x\in X}$ be the ergodic decomposition of $\mu$. Then by \cite[Corollary 3.3]{Chu_Frantzikinakis_Host11}, for $f\in L^\infty(\mu)$ we have that 
    $$\E_\mu(f \mid \mathcal{Z}_{k,\mu})=0 \iff \E_{\mu_x}(f \mid \mathcal{Z}_{k,\mu_x})=0 \text{ for } \mu\text{-almost every } x\in X.$$
    Furthermore
    \begin{equation}\label{prop-CFH}
        f\in L^\infty(\mathcal{Z}_{k,\mu},\mu) \iff f\in L^\infty(\mathcal{Z}_{k,\mu_x},\mu_x) \text{ for }\mu-\text{almost every } x\in X.
    \end{equation}
We will call a measure-preserving system $(Y,\nu,T)$ the \textit{discrete-spectrum factor} of $(X,\mu,T)$, if $(Y,\nu,T)$ is the maximal factor of $(X,\mu,T)$ exhibiting discrete spectrum. This means that if $(Z,\eta,T)$ is another factor of $(X,\mu,T)$ exhibiting discrete spectrum, then $(Z,\eta,T)$ is also a factor of $(Y,\nu,T)$.

\subsection{Background from spectral theory}\label{sec-2.2}

We introduce some background and notation from spectral theory (see \cite{Auslander_Green_Hahn_flows_homogeneous:1963,Katznelson_2004} and \cite{MR2186251,Kanigowski2023} for a classical and modern presentation respectively). We recall that for two Borel measures $\mu$ and $\nu$ on $\mathbb{S}^1$, we say that $\mu$ is \textit{absolutely continuous} with respect to $\nu$, denoted as $\mu\ll \nu$, if $\nu(A)=0\Rightarrow \mu(A)=0$ for each $A\in \mathcal{B}(\mathbb{S}^1).$ If $\mu\ll \nu$ and $\nu\ll \mu$, then we say that $\mu$ and $\nu$ are equivalent, which will be denoted as $\mu\thicksim  \nu$. The equivalence class of all Borel measures equivalent to $\mu$ is called the \textit{type} of $\mu$.
For a unitary operator $U$ on a Hilbert space $\mathcal{H}$, we define the \textit{spectral measure of} $h\in \mathcal{H}$ as the unique finite Borel measure $\sigma_h$ of $\mathbb{S}^1$ provided by the Bochner–Herglotz theorem, satisfying
$$\langle U^{n}h,h\rangle =\int_{\bS^1} z^{n} d\sigma_h(z), ~\forall n\in \Z. $$
By \cite[Section 2.1]{Kanigowski2023}, for each $f\in \mathcal{H}$ the spectral measure $\sigma_f$ is a discrete measure if and only if $f$ belongs to the closure of the linear span of all vectors that are eigenvectors of $U $. The \textit{maximal spectral type of $U$ on $\mathcal{H}$} is the unique finite Borel measure $\sigma$ on $\bS^1$ such that for every $h\in H$, $\sigma_h\ll \sigma$, and for every measure $\nu$ such that $\nu\ll \sigma$ there is $h\in \mathcal{H}$ satisfying $\sigma_h=\nu$. We say that $U$ has discrete spectrum if the maximal spectral type is a discrete measure on $\bS^1$. We say that an operator $U$ has infinite Lebesgue spectrum if $\mathcal{H}$ decomposes into a direct sum of infinitely many pairwise orthogonal closed $U$-invariant subspaces, each of them exhibiting maximal spectral type equivalent to the Haar measure on $\bS^1$.

The following proposition is central in our method and is a direct corollary of \cite[Lemma~1]{Parry}. In order to state it, for a group $G$ we will denote its center by $Z(G)$.
\begin{proposition}\label{PW}
    Let $X=G/\Gamma$ be a nilmanifold. Let $L\leq Z(G)$ be a rational group and $J=L/(L\cap \Gamma)$ be the associated compact abelian group. Then, we have the decomposition 
    $$L^2(X)=\bigoplus_{\chi\in \widehat{J}}V_\chi, $$
    where $\widehat{J}$ is the character group of $J$ and for $\chi\in \widehat{J}$ we denote 
    \begin{equation}\label{def-V-chi}
       V_\chi:=\{ f\in L^2(X) : f(gx)= \chi(g)f(x), \forall g\in L, \text{ for }\mu\text{-a.e. } x\in X\}.  
    \end{equation}
\end{proposition}
For $\alpha\in \bS^1$ and a Borel measure $\sigma$ on $\bS^1$ we will denote $\sigma^\alpha$ the measure given by 
$$\int_{\bS^1} f(z) d\sigma^\alpha(z)=\int_{\bS^1} f(\alpha z) d\sigma(z),~ f\in L^1(\bS^1).  $$
The last proposition of this section will reduce the task of showing that a measure is the Lebesgue measure of $\bS^1$, to showing that the translations of the measure are absolutely continuous with respect to itself. Its proof is essentially contained in the proof of \cite[Proposition 3.2]{MR4797105}, but we include it here for the sake of completeness.
\begin{proposition}\label{Lemma-measure-in-circle}
    Let $\sigma$ be a Borel measure on $\bS^1$ such that for each $\alpha\in \bS^1$, $\sigma^\alpha\ll \sigma$. Then $\sigma\thicksim m_{\bS^1}$.
\end{proposition}
\begin{proof}
    Since $\sigma^\alpha\ll \sigma$ for each $\alpha\in \bS^1$, we have that 
    $$\int_{\bS^1} \sigma^\alpha dm_{\bS^1}(\alpha) \ll \sigma.  $$
    Since the left-hand side defines a rotation-invariant Borel measure on $\bS^1$, it must coincide with the Lebesgue measure $m_{\bS^1}$. On the other hand, by symmetry $\sigma\ll \sigma^\alpha$ for each $\alpha\in \bS^1$. So, by the same argument, we get that $\sigma \ll m_{\bS^1}$, concluding that $\sigma \thicksim  m_{\bS^1}$.
\end{proof}
\subsection{Lie algebras and Mal'cev bases}
We give a brief overview of the tools from the theory of nilpotent Lie groups and Mal'cev bases that we will use. We will base this overview on the contents of \cite[Chapter 10, sections 1.7 and 4.1]{Host_Kra_nilpotent_structures_ergodic_theory:2018}.

Let $G$ be an $s$-step nilpotent Lie group satisfying our hypothesis (i.e. $G^\circ$ is simply connected and $G$ is generated by $G^\circ$ and an elementa $\tau\in G$). The, the exponential map $\exp:\mathfrak{G}\to G^\circ$ is a diffeomorphism from the Lie algebra $\mathfrak{G}$ of $G^\circ$ onto $G^\circ$. We will denote the inverse diffeomorphism of $\exp$ as $\log:G^\circ\to \mathfrak{G}$. For all $\eta,\xi\in \mathfrak{G}$, $\log(\exp(\xi)\exp(\eta))$ is given by the Baker-Campbell-Hausdorff formula:
\begin{equation}\label{BCH}
    \log(\exp(\xi) \exp(\eta))= \xi + \eta + \frac{1}{2} [\xi,\eta] +\frac{1}{12}[\xi,[\xi,\eta]] -\frac{1}{24} [\eta,[\xi,[\xi,\eta]]]+\cdots.
\end{equation}
Since any term with $s$ iterated brackets is trivial, this sum only contains finitely many nonzero terms. Choosing a basis for the vector space $\mathfrak{G}$, the coordinates of $\log(\exp(\xi) \exp(\eta))$ are polynomials of degree at most $s$ in the coordinates of $\xi$ and $\eta$.
We observe that any closed connected subgroup of $G^\circ$ is simply connected \cite[Lemma 10, chapter 10]{Host_Kra_nilpotent_structures_ergodic_theory:2018}.

We have the following result from Mal'cev \cite{Malcev1949}.

\begin{theorem}\label{malcev}
    Let $X=G/\Gamma$ be an $s$-step nilmanifold with connected and simply connected $G$ and assume that $G$ is not an $(s-1)$-step nilpotent Lie group. If $m=\text{dim}(G)$ and $\mathfrak{G}$ is the Lie algebra of $G$, then $\mathfrak{G}$ admits a base $(\xi_1,\ldots,\xi_m)$ satisfying the following properties:
    \begin{enumerate}
        \item The map $\psi:\R^m\to G $ defined by 
$$\psi(t_1,\ldots,t_m)=\exp(t_1\xi_1)\exp(t_2\xi_2)\cdots \exp(t_m\xi_m) $$
is a diffeomorphism from $\R^m$ onto $G$.
\item The image $\psi(\Z^m)$ of $\Z^m$ under $\psi$ is $\Gamma.$
\item For $0\leq j<m$, the linear span $\mathfrak{K}_j$ of $(\xi_{j+1},\ldots,\xi_m)$ is a Lie subalgebra of $\mathfrak{G}$ and $H_j=\psi(\mathfrak{K}_j)$ is a normal Lie subgroup of $G$.
\item Setting $m_i=\text{dim}(G_i)$ for $1\leq i \leq s$, then $G_i=H_{m-m_i}$.
    \end{enumerate}
\end{theorem}
We plan to use \cref{malcev} with $G^\circ$ rather than $G$ in subsequent sections. We will use also the following basic fact from Lie groups theory.
\begin{proposition}\label{connectedness-quotient}
        Let $G$ be a connected and simply connected Lie group. Let $H$ be any normal closed connected subgroup of $G$. Then, $G/H$ is connected and simply connected. 
\end{proposition}
As we were unable to find a proof of \cref{connectedness-quotient} in the literature, we provide one here.
\begin{proof}
    First, since the natural projection $G\to G/H$ is continuous and $G$ is connected and simply connected, then $G/H$ is connected and path connected. Since $H$ is closed, $G\to G/H$ is a fiber bundle. Thus, by \cite[Theorem 4.41 and Proposition 4.48]{HatcherAT}) we get the exact sequence
    $$\pi_1(G)\to \pi_1(G/H)\to \pi_0(H) .$$
    Since $G$ is simply connected we have $\pi_1(G)$ is trivial. Since $H$ is connected, we see that $\pi_0(H)$ is trivial. We conclude that $\pi_1(G/H)$ is trivial and thus $G/H$ is simply connected. 
\end{proof}

\section{Structure of nonergodic nilsystems}
In this section we prove \cref{Theo-A}. We start stating the result of Leibman mentioned in the introduction.
\begin{theorem}[{Cf. \cite[Theorem 2.2]{Leibman03}}]\label{theo2.2L}
    Let $V$ be a connected subnilmanifold of $X$, let $K$ be a connected component of $\pi^{-1}(V)$ and $A$ be a closed subgroup of $G$. There exists a closed subnilmanifold $Y_{V, A}$ of $X$ such that

    \begin{enumerate}[(a)]
        \item for any $x \in V$ one has $\overline{\{ax : a\in A\}} \subseteq a Y_{V, A}$ whenever $g \in K$ is such that $ \pi(g)=x$,
        \item  there exists a zero $\mu_K$-measure set $P \subset K$-where $\mu_K$ denotes a Haar measure on $K$-such that for any $x \in V\setminus \pi(P)$ one has $\overline{\{ax : a\in A\}} =g Y_{V, A}$ whenever $g \in K, \pi(g)=x$.

    \end{enumerate}
We call the subnilmanifold $Y_{V, A}$ the generic orbit for $A$ on $V$; in the case $V=X$ the nilmanifold $Y_{V, A}$ corresponds to the generic orbit for $A$ and will be denoted by $Y_A$.
\end{theorem}
   \begin{remark}
        In the original theorem in \cite{Leibman03} the set $P$ is stated as a \textit{countably polynomial} set. This in particular implies that $P$ is of first category and has measure $0$. Since we only use this last information, we state the theorem with $P$ having measure $0$ only. 
    \end{remark}
    
Let $(X=G/\Gamma,\mu_X,T_\tau)$ be a nilsystem. We will use \cref{theo2.2L} with the connected component $X_0=G^{\circ}\Gamma$ of the projection of the identity $e_X=e_G\Gamma$ in $X$. Let $m\in \N$ be the minimum such that $\tau^mX_0=X_0$. This way, $(X_0,\mu_{X_0},T_{\tau^m} )$ is a connected nilsystem. The connected component of the identity $e_G$ in $\pi^{-1}(X_0)$ is $G^{\circ}$, the connected component of $e_G$ in $G$. Using \cref{theo2.2L} with $V=X_0$ gives a closed subnilmanifold $Y$ such that all but a set of measure zero of $x\in X_0$ satisfy that if $x= g\Gamma$ with $g\in G^{\circ}$, then 
$$\overline{\mathrm{Orb}}_{T_{\tau^m}}(x)=gY. $$
Denote by $H\subseteq G$ the subgroup associated to the subnilmanifold $Y$, i.e. $Y=Hy$ where $y\in X_0$. We will make several reductions. First, we can assume without loss of generality that $y=e_X$. Indeed, let $x\in X_0$ be such that 
$$\overline{\mathrm{Orb}}_{T_{\tau^m}} (x)=gY. $$
Without loss of generality, we can assume $x=e_X$ by a change of base point (see \cite[Chapter 11, section 1.2]{Host_Kra_nilpotent_structures_ergodic_theory:2018}). We will do this assumption throughout the article. Furthermore, we can assume $g=e_G$ by replacing $H$ by $gHg^{-1}$ and $y$ by $gy$. Denote $y=g_y\Gamma$ with $g_y\in G^{\circ}$. As $ \overline{\mathrm{Orb}}_{T_{\tau^m}}(e_X)=Hg_y\Gamma$, we have that $\Gamma\subseteq Hg_y \Gamma$. In particular, this implies that $g_y\in H\Gamma$ which yields $Y=H\Gamma$. Consequently, we can assume that the generic orbit $Y$ is such that
$$Y=\overline{\mathrm{Orb}}_{T_{\tau^m}}(e_X), \text{ and }Y=H\Gamma.$$
Second, we can also assume that for all $g\in G$, $g\tau^{n}g^{-1}\in H$. Indeed, 
let $P\subseteq G^{\circ}$ be with $m_{G^\circ}(P)=0$ given by \cref{theo2.2L} such that for $g\in G^{\circ}\setminus P$ we have $\overline{\mathrm{Orb}}_{T_{\tau^m}}(g\Gamma)= gH\Gamma. $ Thus
$\overline{\mathrm{Orb}}_{T_{g^{-1}\tau^m g}}(e_X)= H\Gamma, $ revealing that $H\Gamma$ is $g^{-1}\tau^m g $-invariant. This extends to all $g\in G^\circ$ by density of $G^\circ\setminus P$ in $G^\circ$. Setting $H'=\langle H \cup \{g^{-1}\tau^{m}g: g\in G\}\rangle$, we see that $H'\Gamma=H\Gamma$. Thus, replacing $H$ by $H'$ we have our claim. 

Finally, for our final reduction, we will see that $H$ can be assumed to be equal to $\langle \tau^m, H^\circ\rangle$, where $H^{\circ}$ denotes the connected component of the identity $e_G$ in $H$. Indeed, on one hand we observed that $[\tau^m,G^\circ]$ is connected, as it is generated by the image of the connected set $G^\circ$ through the continuous function $g\mapsto [\tau^m,g]$. As $[\tau^m,G^\circ]\subseteq H$, we have $[\tau^m,G^\circ]\subseteq H^\circ$. This way $\langle H^\circ,\tau^m\rangle= \langle \tau^m \rangle H^\circ$. On the other hand, since transitivity and ergodicity are equivalent in nilsystems, the system $(Y=H\Gamma,T_{\tau^m})$ is minimal. Hence, the Haar measure $\mu_Y$ on $Y$ gives strictly positive measure to the open set $ H^{\circ}\Gamma$, so there is $l\in \N$ such that
$$ H\Gamma =\bigsqcup_{i=0}^{l-1} (\tau^m)^i H^{\circ}\Gamma.$$
This implies that $H\Gamma=\langle \tau^m,H^\circ\rangle \Gamma $, and since the group $\langle \tau^m,H^\circ\rangle$ clearly contains $ \{g^{-1}\tau^{m}g: g\in G\}$, we can replace $H$ by $\langle \tau^m,H^\circ\rangle$. In this way, we can assume without loss of generality that $H=\langle \tau^m, H^{\circ}\rangle$.

The group $\langle {\tau^m}, H^{\circ}\rangle$ will be called the \textit{Leibman group} associated to the system $(X_0,\mu_{X_0},T_{\tau^m})$. We will show that this group is normal in $G$. To do so, we will need the following result.

\begin{theorem}[{\cite[Theorem 3.5]{Leibman03}}]\label{theo3.5L}
   Let $A$ be a subgroup of $G$ and $Y_A$ be the generic orbit for $A$. The connected components of $Y_A$ are normal subnilmanifolds of $X$.  
\end{theorem}
We will use \cref{theo3.5L} to prove the following.
\begin{proposition}\label{prop-for-H0}
    Let $(X=G/\Gamma,\mu,T_\tau)$ be a nilsystem, $m\in \N$ be minimal such that $\tau^mX_0=X_0$, and $H\subseteq G$ be the Leibman group of $(X_0,\mu_{X_0},T_{\tau^m})$. Then, the connected component of the identity $H^{\circ}$ in $H$ is normal in $G$.
\end{proposition}
\begin{proof}
   By \cref{theo3.5L}, $H^{\circ}$ is a normal subgroup of $G^{\circ}$, so it is enough to show that $H^{\circ}$ is normalized by $\tau$. Notice that $H^{\circ}$ is normalized by $\tau^m$, as $H^\circ$ is normal in $H$. Let $\mathfrak{G}$ be the Lie algebra of $G$ and let $\textbf{Ad}_\tau:\mathfrak{G}\to \mathfrak{G}$ be the automorphism defined by the adjoint representative of $\tau$, i.e. the function induced by $\textbf{Ad}_\tau(g)=\tau g\tau^{-1}$ for $g\in G$. As $\textbf{Ad}_\tau$ is unipotent, we can choose a proper basis in $\mathfrak{G}$ representing $\textbf{Ad}_\tau$ in its Jordan canonical form with $1$'s on the diagonal. Thus, the entries in $\textbf{Ad}_{\tau^n}$ are polynomials in $n$. Let $\mathfrak{H}$ be the Lie algebra of $H$. Then $\tau H^{\circ} \tau^{-1}=H^{\circ}$ is equivalent to $\textbf{Ad}_\tau( \mathfrak{H})=\mathfrak{H}$. As $\textbf{Ad}_{\tau^{nm}}( \mathfrak{H})=\mathfrak{H}$ for each $n\in \Z$, we conclude that $\textbf{Ad}_{\tau^{n}}( \mathfrak{H})=\mathfrak{H}$ for each $n\in \Z$ as this is a polynomial condition with infinite zeros. Thus, we conclude that $\tau H^{\circ} \tau^{-1}=H^{\circ}$, concluding that $H^{\circ}$ is normal in $G$.
\end{proof}

Now, with \cref{prop-for-H0} at hand, we are able to show that the Leibman group is also normal in $G$.

\begin{proposition}\label{H-normal}
         Let $(X=G/\Gamma,\mu,T_\tau)$ be a nilsystem, $m\in \N$ be minimal such that $\tau^mX_0=X_0$, and $H\subseteq G$ be the Leibman group of $(X_0,\mu_{X_0},T_{\tau^m})$. Then $H$ is normal in $G$.
\end{proposition}
\begin{proof}
Since $[\tau,G]=[\tau,G^\circ]\subseteq H^\circ$, we have that for $g\in G^{\circ}$,
$$gHg^{-1}= g\langle \tau^m, H^{\circ}\rangle g^{-1}=\langle g \tau^m g^{-1},gH^{\circ}g^{-1}\rangle =\langle \tau^m , H^{\circ}\rangle= H,$$
where we used the fact that $H^{\circ}$ is normalized by $G^{\circ}$ and that $g\tau^m g^{-1}H^{\circ}=\tau^m H^{\circ}$. That being so, it is enough to show that $H$ is normalized by $\tau$ to conclude. This is obtained by noticing that
   $$\tau H \tau^{-1}= \tau \langle \tau^m, H^{\circ}\rangle \tau^{-1} = \langle \tau \tau^m\tau^{-1} , \tau H^{\circ}\tau^{-1} \rangle =\langle \tau^m, H^{\circ}\rangle = H.$$
   \end{proof}
We define the \textit{Leibman group} associated with the nilsystem $(X,\mu,T)$ as $H:=\langle  H^{\circ},\tau\rangle$, where $H^{\circ}$ is the connected component of the Leibman group of $(X_0,\mu_{X_0},T^m)$. Observe that no confusion arises from using $H^{\circ}$ to denote the connected component of the Leibman group of $X$ and the connected component of the Leibman group of $X_0$, since both of these groups coincide. Now we prove \cref{Theo-A}.
\begin{proposition}\label{first-part-theo-A}
    Let $(X=G/\Gamma,\mu, T)$ be a nilsystem and $H$ its Leibman group. Then, $H$ is normal in $G$ and for all $g\in G^{\circ}$ except for a set of $m_{G^{\circ}}$-measure zero,     $$\overline{\mathrm{Orb}}_T(g\Gamma)=gH\Gamma. $$
    In particular, $H$ is a rational subgroup of $G$.
\end{proposition}
\cref{first-part-theo-A} is the reason we call $H$ the Leibman group of a nilsystem $(X=G/\Gamma,\mu,T)$, since it describes the orbit of almost every point in $X$, extending \cref{theo2.2L} in the case $A=\overline{\langle \tau \rangle}$.
\begin{proof}
Denote by $H'$ the Leibman group of $(X_0,\mu_{X_0},T^m)$. First we prove that $H$ is normal in $G$. Let $l\in \N$ such that $\tau^{ml}H^{\circ}\Gamma =H^{\circ}\Gamma$. Since for each $n\in \N$, $[G^{\circ},\tau^{nm}]\subseteq H^{\circ}$, then we have that $\log([G^{\circ},\tau^{nm}])$ is contained in the Lie algebra $\mathfrak{H}$ associated to $H^{\circ}$. Observe that $\mathfrak{H}$ is a subspace of $\mathfrak{G}$ which is defined by linear equations equations. On the other hand, by the Baker–Campbell–Hausdorff formula \cref{BCH}, for each $g\in G^{\circ}$ 
$\log([g,\tau^{n}])$ are polynomial expressions on $n$, such that $\log([g,\tau^{mn}])$ satisfies the linear equations defining $\mathfrak{H}$ for each $n$. This implies that $\log([g,\tau^n])$ belongs to $\mathfrak{H}$ for each $n\in\N$, concluding that $[\tau,G^{\circ}]\subseteq H^{\circ}$. As well as in the proof of \cref{H-normal}, this implies that $H=\langle  H^{\circ}, \tau\rangle$ is normal in $G$.

Let $A\subseteq G^{\circ}$ be the $m_{G^{\circ}}$-measure zero set given by \cref{theo2.2L} such that for all $g\in G^{\circ}\setminus A$,
    $$\overline{\mathrm{Orb}}_{T^m}(g\Gamma)=gH'\Gamma. $$
    This implies that for $g\in G^{\circ}\setminus A$,
    $$ \overline{\mathrm{Orb}}_{T}(g\Gamma)=\bigcup_{i=0}^{m-1}\tau^i gH'\Gamma=g \bigcup_{i=0}^{m-1}(g^{-1}\tau g)^iH' \Gamma =g \bigcup_{i=0}^{m-1}\tau^iH' \Gamma =g H\Gamma .$$
\end{proof}
Now we prove the second part of \cref{Theo-A}.
\begin{theorem}\label{Theo-A-part-2}
Let $(X=G/\Gamma,\mu,T_{\tau})$ be a nilsystem and $H\subseteq G$ its Leibman group. Then, for each $k\in \N_0$, a function $f\in L ^2(X)$ is $\mathcal{Z}_k$-measurable if and only if it is $H_{k+1}$-invariant, namely
    \begin{equation*}
       L^2(X,\mathcal{Z}_k) \cong L^2(H_{k+1}\backslash X,\mathcal{B}(H_{k+1}\backslash X)).
    \end{equation*}
\end{theorem}
\begin{proof}
Let $k\in \N_0$ be a nonnegative integer. By \cref{first-part-theo-A} we have that $\mu$-a.e. $x=\tau^{i}g\Gamma\in X$ with $g\in G^{\circ}$ and $i\in \N_0$,
\begin{equation}\label{Orbit-generic-point}
    \overline{\mathrm{Orb}}_{T_\tau}(x)= \overline{\mathrm{Orb}}_{T_\tau}(g\Gamma)=  gH \Gamma .
\end{equation}

Since $H_{k+1}$ is a normal rational subgroup of $G$, the nilmanifold $ G/H_{k+1}\Gamma$ is well defined and is a factor of $X$. Let $f\in L^2(X)$ be a function. First, by \cref{prop-CFH}, 
$$f \text{ is }\mathcal{Z}_k(X)\text{-measurable if and only if }f \text{ is } \mathcal{Z}_{k,\mu_x}(X) \text{-measurable for } \mu\text{-a.e. }x\in X.$$
 Second, since $H_{k+1}$ is normal, by \cref{Orbit-generic-point} we have $\mu$-a.e. $x=\tau^{i} g\Gamma\in X$, $f$ is $\mathcal{Z}_{k,\mu_x}$ measurable if and only if $f$ is $H_{k+1}$-invariant, concluding.
\end{proof}
\section{Spectrum of nonergodic nilsystems}
We will now turn to the representation of the discrete-spectrum factor. Given a not necessarily ergodic nilsystem $(X=G/\Gamma,\mu_X,T)$, its discrete-spectrum factor can be larger than $G/G_2\Gamma$. A trivial example comes from taking the transformation $T$ as left rotation by an element in the center of the group $\tau\in Z(G)$. In this case, $(X=G/\Gamma,\mu_X,T)$ is a system of order $1$ with pure discrete spectrum. Moreover, the factor of order $1$ and the discrete-spectrum factor of $(X=G/\Gamma,\mu_X,T)$ coincide. This is not always the case for nonergodic systems, as the following example in a $2$-step nilsystem shows.
\begin{example}\label{example-2.1}
    Consider $(\T^2,\mu,T)$ where $T(x,y)=(x,x+y)$. This transformation is clearly nonergodic (with ergodic components of $\mu$ given by $\delta_x\otimes m$) and of order $1$. We notice that $L^2(X,\mu, \mathcal{Z}_0)$ corresponds to the space generated by $((x,y)\to e(px))_{p\in \Z}$. 
    
    It turns out that $L^2(X,\mu,\mathcal{Z}_0)$ actually coincides with the discrete-spectrum factor of $(\T^2,\mu,T)$. Indeed, for this observe that the system $(\T^2,T)$ posses an orthonormal basis $((x,y)\to e(px+qy))_{p,q\in \Z}$. Thus, it is enough to see that the functions $(x,y)\to e(px)$ present discrete spectrum, while the functions $(x,y)\to e(qy)$ present continuous spectrum. Indeed, if we denote $f(x,y)=e(px)$ then
    $$\int \overline{f} \cdot T^n f d\mu = \int e(-px) e(px) d\mu(x,y)= 1,   $$
    showing that $f$ exhibits discrete spectrum. 
    On the other hand, if we denote $g(x,y)=e(qy)$ then we have that for $n\in \Z\setminus\{0\}$
    $$\int \overline{g} \cdot T^n g d\mu = \int e(-qy) e(qy + nqx) d\mu(x,y) = \int e(nqx) dm_\T(x) = \ind{q=0},  $$
    which implies that $g$ exhibits Lebesgue spectrum. Thus, we get the following decomposition 
    $$L^2(\T^2)= L^2(\{0\}\times \T \backslash \T^2)\oplus L^2(\T\times \{0\} \backslash \T^2) \oplus \{0\}, $$
    where $L^2(\{0\}\times \T \backslash \T^2)=L^2(X,\mu,\mathcal{Z}_0)$ is the discrete-spectrum factor of $(\T^2,\mu,T)$, $ L^2(\T\times \{0\} \backslash \T^2)$ is the weak-mixing nonuniform component and exhibits infinite Lebesgue spectrum, and $\{0\}$ is the trivial uniform component of the space.  
\end{example}
The previous example shows that a more delicate decomposition plays an important role to identify the discrete-spectrum factor of a nonergodic nilsystem. The general idea will be to split $L^2(X)$ in three:
$$L^2(X)= L^2( \mathcal{J}([\tau,G],\Gamma)\backslash X) \oplus \left( L^2( \mathcal{J}([\tau,G],\Gamma)\backslash X)^\perp \cap L^2([H,H]\backslash X)\right) \oplus L^2([H,H]\backslash X)^\perp,$$
where we will see that $L^2( \mathcal{J}([\tau,G],\Gamma)\backslash X)$ is the discrete-spectrum factor of $(X,\mu,T)$, and thus will contain the discrete spectrum of the system; $\left( L^2( \mathcal{J}([\tau,G],\Gamma)\backslash X)^\perp \cap L^2([H,H]\backslash X)\right) $ is the weak-mixing nonuniform component and will exhibit infinite Lebesgue spectrum; and finally $L^2([H,H]\backslash X)^\perp$ is the uniform component and will exhibit infinite Lebesgue spectrum as well. The study of the spectrum of $L^2( \mathcal{J}([\tau,G],\Gamma)\backslash X)$ and $L^2([H,H]\backslash X)^\perp$ will rely in similar techniques to those of \cite{MR4797105}. However, when it comes to study the spectrum on the weak mixing nonuniform component, i.e. $ L^2( \mathcal{J}([\tau,G],\Gamma)\backslash X)^\perp \cap L^2([H,H]\backslash X)$, the techniques from \cite{MR4797105} are not enough. Instead, we rely in structural results of system of order 1 from Host and Frantzikinakis \cite{FRANTZIKINAKIS_HOST_2018}, which allows to study the spectrum through Mal'cev coordinates. 

\subsection{The discrete-spectrum and uniform components}

We start proving that the orthocomplement of the $\mathcal{Z}_1$-factor possesses purely infinite Lebesgue spectrum. 
\begin{theorem}\label{Spectrum-uniform-part}
    Let $(X=G/\Gamma,\mu,T)$ be a nilsystem and $H\subseteq G$ its Leibman group. Then $T$ exhibits infinite Lebesgue spectrum in $L^2([H,H] \backslash X)^\perp$.
\end{theorem}
\begin{proof}
    Let $f\in L^2([H,H] \backslash X)^\perp$. By \cref{Theo-A-part-2}, for $\mu$-a.e. $x\in X$, $f$ is $Z_{1,\mu_x}$-uniform. In particular, by \cite[Theorem 1.4]{MR4797105} if $\sigma_{f,\mu_x}$ denotes the spectral measure of $f$ in $L^2(X,\mu_x)$, then we have that $\sigma_{f,\mu_x}\ll m_{\bS^1}$ for $\mu$-a.e. $x\in X$. We observe that if $\sigma_f$ is the spectral measure of $f$ in $L^2(X,\mu)$, then
    $$ \sigma_f = \int_X \sigma_{f,\mu_x} d\mu(x).$$
    Let $A\subseteq \bS^1$ be a Borel set such that $m_{\bS^1}(A)=0$. Since $\mu$-a.e. $x\in X$, $\sigma_{f,\mu_x}\ll m_{\bS^1}$, we have that $\mu$-a.e. $x\in X$, $\sigma_{f,\mu_x}(A)=0$. This implies that $\sigma_f(A)=0$, concluding that $\sigma_f\ll m_{\bS^1}. $
    
    To prove that the maximal spectral type on $L^2(X)$ is countable and achieved, it is enough to study a factor of $X$. By replacing $X$ by $G_3\backslash X$, we reduce to the case in which $X$ is a $2$-step nilsystem. Hence, the group $[H,H]$ is rational group which lies in $Z(G)$. Using \cref{PW} with the group $L:=[H,H]$ we have that
    \begin{equation}\label{eq-countable-split}
     L^2([H,H]\backslash X)^\perp=\bigoplus_{\chi \in \widehat{J}\setminus\{1\}} V_\chi,    
    \end{equation}
    where $J=L/(L\cap \Gamma)$ and $V_\chi$ are defined as in \cref{def-V-chi}. Let $\chi \in \hat{L}\setminus\{1\}$ such that $\chi(L\cap \Gamma)\equiv 1$. We will show that $V_\chi$ has maximal spectral type the Lebesgue measure of $\bS^1$. This is enough to conclude by \cref{eq-countable-split}. First, we prove that $\chi([H,\tau])= \bS^1$. Assume by contradiction that this is not true. The group $[H,\tau]=[H^\circ,\tau]$ is connected by being a group generated by the image of the connected set $H^\circ$ through the continuous function $h\mapsto [h,\tau]$. This being so, by continuity of $\chi$, we have that $\chi([H,\tau])=\{1\}$, as $\{1\}$ and $\bS^1$ are the two only connected subgroups of $\bS^1$. Let $h,g\in H$, by minimality we have that there are sequences $(n_i)_{i\in \N},(m_i)_{i\in \N}\subseteq \N$ and $(\gamma_{i})_{i\in \N},(\zeta_{i})_{i\in \N}\subseteq \Gamma\cap H$ such that $\tau^{n_i}\gamma_{i}\to h$ and $\tau^{m_i}\zeta_{i}\to g,$ as $i\to \infty.$ Since $\chi([H,\tau])=\{1\}$, we get
    \begin{align*}
        \chi([h,g])= \lim_{i\to\infty} \chi([\tau^{n_i}\gamma_i,\tau^{m_i}\zeta_i]) =\lim_{i\to\infty} \chi([\gamma_i,\zeta_i]) =1.
    \end{align*}
    We conclude that $\chi\equiv 1$, which is a contradiction. Thus, $\chi([H,\tau])= \bS^1$.

    Let $f\in V_\chi$ which spectral measure attains the maximal spectral type $\sigma$. For $h\in H$ denote $\alpha:=  \chi([h,\tau])$. We have that 
    \begin{align*}
       \int \overline{T_hf} \cdot T_\tau^nT_hfd\mu&= \int \overline{f}(hx)\cdot  f(h\tau^nx) d\mu(x)= \chi([h,\tau])^n\int \overline{f}\cdot T_\tau^nf d\mu\\
       &= \alpha^n \int_{\bS^1} z^n d\sigma(z)= \int_{\bS^1} z^n d\sigma^\alpha(z),
    \end{align*}
where we recall that $\sigma^\alpha$ denotes the measure $\sigma$ translated by $\alpha$. Thus, $\sigma^\alpha\ll \sigma$ for all $\alpha\in \chi([H,\tau])=\bS^1    $. By \cref{Lemma-measure-in-circle} we conclude that $\sigma=m_{\bS^1}$, finishing the proof.  
\end{proof}

Now we show that $T$ exhibits pure discrete spectrum in $L^2(\mathcal{J}([\tau,G],\Gamma)\backslash X)$. 
\begin{theorem}\label{main-theorem}
    Let $(X=G/\Gamma,\mu,T)$ a nilsystem. Then $T$ exhibits pure discrete spectrum on $L^2(\mathcal{J}([\tau,G],\Gamma)\backslash X)$. 
\end{theorem}
\begin{proof}
    By replacing $G$ by $G/\mathcal{J}([\tau,G],\Gamma)$ and $\Gamma$ by $ (\mathcal{J}([\tau,G],\Gamma)\Gamma/ (\mathcal{J}([\tau,G],\Gamma))$, we just need to prove that if $\tau\in Z(G)$, then $(X=G/\Gamma,\mu,T)$ exhibits pure discrete spectrum. Using \cref{PW} with $L=Z(G)$, we obtain the decomposition 
    $$ L^2(X)= \bigoplus_{\chi \in \widehat{J}} V_\chi,$$
    where $J=Z(G)/ (Z(G)\cap \Gamma)$ and for each character $\chi$ in $Z(G)$ with $\chi(Z(G)\cap \Gamma)\equiv 1$, $V_\chi$ is defined as in \cref{def-V-chi}. For $f\in V_\chi$, we have that the Fourier coefficients of its spectral measure $\sigma_f$ are given by 
    $$\widehat{\sigma_f}(n)=\int \overline{f}\cdot  T^n f d\mu= \chi(\tau)^n \int |f|^2 d\mu,$$
    which are the Fourier coefficients of a discrete measure, whence the result follows. 
\end{proof}

\subsection{Weak mixing and nonuniform component}
As we mentioned previously, it is not possible to use similar techniques to the ones of \cite{MR4797105} to show that the weak mixing nonuniform component exhibits pure infinite Lebesgue spectrum.
The main issue is that there may not be a nontrivial element of the form $[\tau,g]$ for $g\in G^\circ$ lying in the center $Z(G)$ of $G$, which was crucial in \cite{MR4797105} and for us in \cref{Spectrum-uniform-part} and \cref{main-theorem} in order to use \cref{PW}.
\begin{example}
 Let
\[
G = \left\{ \begin{pmatrix}
1 & x & y & z  \\
0 & 1 & u & v \\
0 & 0 & 1 & w \\
0 & 0 & 0 & 1
\end{pmatrix} \mid x, y, z,u,v,w \in \R \right\}
\]
be the group of $4\times 4$ upper triangular matrices with $1$s in the diagonal, with the usual matrix multiplication. Then \( (G, \cdot) \) is a Lie group, known as the $4$-dimensional Heisenberg group. We consider a Heisenberg system given by the triplet \((X=G / \Gamma, m_X, T)\), where \(\Gamma=G \cap M_{4}(\Z)\) is a discrete cocompact subgroup of \(G\), \( m_X \) denotes the Haar measure on the quotient space \( G / \Gamma \), and \( T \) represents the transformation induced by left multiplication by the element
   $$\tau =  \begin{pmatrix}
1 & 0 & y_\tau & 0 \\
0 & 1 & u_\tau & 0 \\
0 & 0 & 1 & 0 \\
0 & 0 & 0 & 1
\end{pmatrix},$$
with $y_\tau,u_\tau\in(0,1)$. We observe that 
$$Z(G)= \left\{ \begin{pmatrix}
1 & 0 & 0 & z  \\
0 & 1 & 0 & 0 \\
0 & 0 & 1 & 0 \\
0 & 0 & 0 & 1
\end{pmatrix} \mid z \in \R \right\}.$$
We have that if
$$g=\begin{pmatrix}
1 & x & y & z  \\
0 & 1 & u & v \\
0 & 0 & 1 & w \\
0 & 0 & 0 & 1
\end{pmatrix} \in G, \text{ then } [\tau,g]= \begin{pmatrix}
1 & 0 & -u_\tau x & w(u_\tau x+ y_\tau)  \\
0 & 1 & 0 & u_\tau w \\
0 & 0 & 1 & 0 \\
0 & 0 & 0 & 1
\end{pmatrix}. $$
It follows that $[\tau,g]\in Z(G)$ if and only if $x=w=0$, which would imply that $[\tau,g]=e_G$, showing that $\{[\tau,g] : g\in G\}\cap Z(G)= \{e_G\}$. 
\end{example}
Thus, we need a different strategy to tackle the weak mixing nonuniform case: we will use the existence of Mal'cev bases of connected simply connected nilpotent Lie groups. For this, we will need the following propositions. 
\begin{proposition}\label{app-coarea-formula}
    Let $f\in \mathcal{C}^1(\R^d,\R)$ be such that the set $\{x\in \R^d : \nabla f (x)=0\}$ has zero Lebesgue measure. Then the pushforward $f^*(m_{\R^d})$ of the Lebesgue measure is absolutely continuous with respect the Lebesgue measure $m_{\R}$ on $\R$.
\end{proposition}
We were unable to find a proper source for this proposition, so we provide a proof for the sake of completeness. 
\begin{proof}
    Let $A\subseteq \R$ be a set of zero Lebesgue measure. We want to see that 
    $$\int_{\R^d} \ind{ A}(f(x)) dm_{\R^d}(x)=0. $$
    Since $f\in \mathcal{C}^1(\R^d,\R)$, the set  $(\nabla f)^{-1}(\{0\})$ is closed. Thus, we can cover its complement with countably many open balls compactly contained therein. Let $B$ be one of these balls. It is enough to prove that 
    $$ \int_B \ind{A}(f(x)) dm_{\R^d}(x)=0.$$
    The function $f|_B: B \to \R$ is Lipschitz as $f$ is $C^1$ and $\overline{B}$ is compact in $\R^d$. In addition, $|\nabla f|$ restricted to $B$ is bounded and bounded away from zero. By the Coarea formula, if we denote $H_{d-1}$ the $(d-1)$-dimensional Hausdorff measure, we have that for all $g\in L^1(B,m_{\R^d})$,
    $$\int_B g(x) |\nabla f(x)| dm_{\R^d}(x) = \int_{\R} \left( \int_{f^{-1}(t)} g(x) d H_{d-1}(x) \right) dm_{\R}(t).  $$
    Consider the $L^1(B,m_{\R^d})$ function $g(x)= \ind{B}(x) \ind{A}(f(x)) \cdot /|\nabla f (x)|$. Then, we have that 
    \begin{align*}
      \int_B \ind{A}(f(x)) dm_{\R^d}(x)&= \int_{\R} \left( \int_{f^{-1}(t)} \ind{B}(x) \ind{A}(f(x)) dH_{d-1} (x) \right) dm_{\R}(t)\\
      &=  \int_{A}  H_{d-1}(f^{-1}(t) \cap B) dm_{\R}(t)=0
    \end{align*}
where we used that $A$ has zero Lebesgue measure, finishing the proof. 
\end{proof}
\begin{proposition}\label{tj-explicit}
    Let $(X=G/\Gamma,\mu,T=T_\tau)$ be a nilsystem of order $1$, and $H$ its Leibman group. Let $Y=G/H\Gamma$ be the factor defined by the invariant sigma algebra $\mathcal{I}(T)$ of $X$ and $(\mu_y)_{y\in Y}$ the ergodic desintegration given by it. Let $\eta_{G}: \R^d \to G^\circ$ be the Mal'cev coordinates of $G^\circ$ adapted to $H^\circ$. Let $(\chi_j)_{j\in \N}\subseteq \widehat{H}$ be an orthonormal basis of $L^2(H/H\cap \Gamma)$. Then, there is $l\leq d$ such that for each $j\in \N$ the function $t_j:Y\to \T$ defined for $\mu_Y$-a.e. $y=gH\Gamma$ with $g\in \eta_G(\{0\}^{d-l}\times [0,1)^l)$, by $$e(t_j(y))= \chi_j(g^{-1}\tau g)$$
    is a Borel map such that  for $\mu_Y$-a.e. $y\in Y$, $\mu_y(\{e(t_j(y))\}_{j\in \N})=1$. 
\end{proposition}
\begin{proof}
Since $G$ is generated by $G^\circ$ and $\tau\in H$, we have that $G/H$ is the image of $G^\circ$ under the continuous projection $g\in G\mapsto  gH \in G/H$. Thus, the nilpotent Lie group $G/H\cong G^\circ /(G^\circ \cap H)$ is connected. Moreover, as $\tau$ generates a free abelian group with trivial intersection with $G^\circ$, we have that $G^\circ \cap H= H^\circ$. Thus, the nilmanifold $$Y=(G/H)/(H\Gamma/H)\cong(G^\circ /H^\circ)/(H^\circ (\Gamma\cap G^\circ)/H^\circ)$$ is a connected nilmanifold of dimension $l=\text{dim}(G^\circ) - \text{dim}(H^\circ)$ with $(G^\circ /H^\circ)$ connected simply connected by \cref{connectedness-quotient}. We choose a Mal'cev basis in $G^\circ$ adapted to $H^\circ$, meaning that the fundamental domain of $H^\circ$ is identified with $[0,1)^{d-l}\times \{0\}^l$. Thus, the fundamental domain of $G^\circ/H^\circ$ is identified with $F=\{0\}^{d-l}\times [0,1)^l$. Let $K:=\eta_G(F)$. In particular, $\mu_Y$ is the image of $\ind{K} m_G$ through $\eta_G$ and $\psi : F\to Y$ given by $\psi( x) = \eta_G(x)H\Gamma, $ is continuous and bijective between Polish spaces. By Lusin–Souslin's theorem, its inverse $\psi^{-1}:Y\to F$ is Borel. Therefore, we get
$$e(t_j(y))= \chi_j( (\psi(y))^{-1}\tau \psi(y)) $$
which is Borel, and thus so is $t_j$.

To finish, we need to show that $\mu_Y$-a.e. $y\in Y$, $\mu_y(\bS^1\setminus \{e(t_j(y))\}_{j\in \N})=0$. For this, we observe that for $j\in \N$, the function $\chi_j^g(h) =\chi_j(g^{-1}hg) $ is such that $\chi_j^g(H\cap g\Gamma g^{-1})\equiv 1$. Since $(\chi_j)_{j\in \N}$ forms an orthonormal basis of $L^2(H/H\cap \Gamma)$, it follows that $(\chi_j^g)_{j\in \N}$ forms an orthonormal basis of characters of $L^2(H/(H\cap g\Gamma g^{-1}))$. Since for $\mu_Y$-a.e. $y\in Y$ there is $g\in G^\circ$ such that $y=gH\Gamma=\overline{\mathrm{Orb}}_{T}(g\Gamma)$ which is isomorphic to $H/(H\cap g\Gamma g^{-1})$, we have that the eigenvalues of $T$ in $H/(H\cap g\Gamma g^{-1})$ are the atoms in which the discrete measure $\mu_y$ is supported. These eigenvalues are given by
$$\{\chi_j^g(\tau) \}_{j\in \N} = \{\chi_j(g^{-1}\tau g) \}_{j\in \N}= \{ e(t_j(y)) \}_{j\in \N},$$
concluding the result. 
\end{proof}
We will now recall some definitions from \cite[Section 5]{FRANTZIKINAKIS_HOST_2018}. We remark that we treat spectral measures as measures on $\bS^1$, while in \cite[Section 5]{FRANTZIKINAKIS_HOST_2018} they are treated as measures in $\T$, whence the necessity of composing the functions $t_j$ with the natural identification $x\in \T\mapsto e(x)\in \bS^1$ in \cref{tj-explicit}.
\begin{definition}\label{localization-lemma}
    Let $(X,\mu,T)$ be a measure-preserving system. A relative orthonormal system (with respect to the $T$-invariant $\sigma$-algebra $\mathcal{I}(T)$) is a countable family $(\phi_j)_{j\in \N}$ of functions belonging to $L^2(\mu)$ such that: 
    \begin{enumerate}
        \item  for every $j \in \mathbb{N}$, for $\mu$-almost every $x\in X$, $ \mathbb{E}(\left|\phi_j\right|^2 \mid \mathcal{I}(T))(x)$ has value zero or one; and
        \item for all $j, k \in \mathbb{N}$ with $j \neq k$, for $\mu$-almost every $x\in X$, $ \mathbb{E}(\phi_j \overline{\phi_k} \mid \mathcal{I}(T))(x)=0.$  
    \end{enumerate}
The family $\left(\phi_j\right)_{j \in \mathbb{N}}$ is a relative orthonormal basis if it also satisfies:
\begin{enumerate}
\item[\textbf{3.}] the linear space spanned by all functions of the form $\phi_j \psi$, is dense in $L^2(\mu)$, where $j \in \mathbb{N}$ and $\psi \in L^{\infty}(\mu)$ varies over all $T$-invariant functions.
    \end{enumerate}
\end{definition}

\begin{definition}
Let $(X,\mu,T)$ be a measure-preserving system. Let $\lambda \in L^{\infty}(\mu)$ be a $T$-invariant function and $\phi \in L^{\infty}(\mu)$. We say that $\phi$ is an eigenfunction with eigenvalue $\lambda$ if:
  \begin{enumerate}
      \item  for $\mu$-almost every $x \in X$, $|\phi(x)|$ has value zero or one;
\item  for $\mu$-almost every $x \in X$ such that $\phi(x)=0$, $\lambda(x)=0$; and
\item  for $\mu$-almost every $x\in X$, $\phi \circ T(x)=\lambda \cdot \phi(x)$.
  \end{enumerate}
\end{definition}
The following lemma is an immediate consequence of the proof of \cite[Lemma 5.8]{FRANTZIKINAKIS_HOST_2018} and \cite[Theorem 5.2]{FRANTZIKINAKIS_HOST_2018}. Hence, we omit the proof of \cref{relative-basis} as the argument is identical, except that the use of \cite[Proposition 5.6]{FRANTZIKINAKIS_HOST_2018} in the proof of \cite[Lemma 5.8]{FRANTZIKINAKIS_HOST_2018} is omitted and instead assumed as a hypothesis that the sequence $\{t_j\}_{j\in \N}$ constructed there is given.

\begin{lemma}\label{relative-basis}
    Let $(X,\mu,T)$ be a system of order one and $f\in L^2(\mu)$. Let $(Y,\mathcal{Y},\nu,T)$ be the factor defined by the invariant sigma algebra $\mathcal{I}(T)$, where $(\mu_y)_{y\in \N}$ denotes the ergodic decomposition given by this factor. If there are Borel maps $t_j:Y\to \T$ for $j\in \N$, such that for $\mu_Y$-a.e. $y\in Y$, the Haar measure $\mu_y$ on $\overline{\mathrm{Orb}}_{T}(y)$ is supported in $\{t_j(y)\}_{j\in \N}$, then there exists a relative orthonormal basis of eigenfunctions $(\phi_j)_{j\in\N}$ with eigenvalues $\lambda_j(y) \in\{\ind{A_i}(y) e(t_i(y))\}_{i\in \N}$ where 
    $$A_i=\{y\in Y : \mu_y(\{e(t_i(y))\})>0 \}. $$
\end{lemma}

Now we prove \cref{Theo-B}, which we restate now. 
\begin{theorem}\label{spectral-decomposition}
Let $(X=G/\Gamma,\mu,T)$ be a nilsystem. Then the Koopman representation of $T$ in $L^2(X,\mu)$ decomposes as
  $$ L^2(X)=L^2(\mathcal{J}([\tau,G],\Gamma)\backslash X)\oplus L^2(\mathcal{J}([\tau,G],\Gamma)\backslash X)^\perp,$$
  where $T$ exhibits discrete spectrum on $L^2(\mathcal{J}([\tau,G],\Gamma)\backslash X)$ and infinite Lebesgue spectrum on $L^2(\mathcal{J}([\tau,G],\Gamma)\backslash X)^\perp$ if nontrivial.
\end{theorem}
\begin{remark}\label{theo-B-erg}
When $(X=G/\Gamma,\mu,T=T_\tau)$ is ergodic, \cref{Theo-B} implies the known result that $L^2([G,G]/X)$ is the Kronecker factor of $(X=G/\Gamma,\mu,T=T_\tau)$. Indeed, we claim $\mathcal{J}([\tau,G],\Gamma)=[G,G]$. The inclusion $\mathcal{J}([\tau,G],\Gamma)
\subseteq [G,G]$ is direct from the fact that $[G,G]$ is connected, rational, and normal, and also contains $[\tau,G]$. For the other inclusion, observe that it is enough to prove that in $G/\mathcal{J}([\tau,G],\Gamma)$, the group $[G,G]$ is trivial, or equivalently $[G^{\circ},G^{\circ}]=\{e_G\}$. So, let us assume that $[\tau,G]=\{e_G\}$. Take $g,h\in G^{\circ}$. By minimality, there are sequences $(n_i)_{i\in \N},(m_i)_{i\in \N}\subseteq \N$ and $(\gamma_i)_{i\in \N}, (\xi_{i})_{i\in \N}\subseteq \Gamma \cap G^{\circ}$ such that $\tau^{n_i}\gamma_i\to g$ and $\tau^{m_i}\xi_i\to h$ as $i\to \infty.$ Then, we see that by the continuity of the brackets $[\cdot,\cdot]$, $[g,h]$ is the limit as $i\to \infty$ of $[\tau^{n_i}\gamma_i,\tau^{m_i}\xi_i]= [\gamma_i,\xi_i]\subseteq \Gamma. $ As $\Gamma$ is closed, we conclude that $[G^{\circ},G^{\circ}]\subseteq \Gamma$. As $[G^{\circ},G^{\circ}]$ is connected as $G^\circ$ and the brackets $[\cdot,\cdot]$ are continuous, we conclude that $[G^{\circ},G^{\circ}]=\{e_G\}$.
\end{remark}

\begin{proof}[Proof of \cref{Theo-B}]
Let $H$ be the Leibman group of $(X=G/\Gamma,\mu,T)$. First, we notice that if $[\tau,G]=\{e_G\}$ then $L^2(X)=L^2(\mathcal{J}([\tau,G],\Gamma)\backslash X)$ and the result follows from \cref{main-theorem}. Thus, we will assume that $[\tau,G]\neq \{e_G\}$. Second, by \cref{Spectrum-uniform-part}, $L^2([H,H]\backslash X)^\perp$ exhibits infinite Lebesgue spectrum. Hence, it is enough to prove the result in $L^2([H,H]\backslash X)$. Then, replacing $X$ by $[H,H]\backslash X$, we assume without loss of generality that $[H,H]=\{e_G\}$ and thus $H$ is an abelian.

Let $(\chi_j)_{j\in \N}$ be a set of characters in $\widehat{H}$ forming an orthonormal basis of $L^2(H/(H\cap \Gamma))$, such that $\chi_j(\Gamma\cap H)\equiv 1$. Denote by $F= \{0\}^{d-l} \times [0,1)^l$ the fundamental domain for the projection of $G^\circ/H^\circ$ onto $Y$. Using  \cref{relative-basis}  with the family of functions $\{t_j\}_{j\in \N}$ given by \cref{tj-explicit}, we obtain the decomposition
    $$L^2(X)= \bigoplus_{j\in \N} \overline{L^{\infty}(X,\mathcal{I}(T)) \phi_j  }, $$
    where for each $j\in \N$, $\phi_j$ has eigenvalue $\lambda_j(y)=  \ind{A_{i_j}}(y) e(t_{i_j}(y))$ where $A_{i_j}=\{y \in Y : \mu_y(\{e(t_{i_j}(y))\})>0 \}.$ For the sake of notation, let us denote $i_j=j$. Thus, it is enough to see that the spaces $E_j:=  \overline{ L^\infty(X,I(T))  \phi_j} $ have maximal spectral type absolutely continuous with respect to the Lebesgue measure on $\bS^1$, whenever $\phi_j$ is not $\mathcal{J}([\tau,G],\Gamma)$-invariant. Let $\psi \in L^\infty(X,I(T))$ and $j \in \N$ be such that $\psi\phi_j\not\equiv 0$ and $\phi_j$ are not $\mathcal{J}([\tau,G],\Gamma)$-invariant. We have 
$$\int_X (\overline{\psi\phi_j})\cdot T^n(\psi\phi_j) d\mu=\int_Y |\psi \phi_j|^2 \ind{A_j} e(t_j(y))^nd\mu_Y. $$
In this way, the spectral type of $f= \psi\phi_j$ is given by the pushforward of $|\psi \phi_j|^2\ind{A_j}\mu_Y  $  through the map $y\in Y \mapsto e(t_j(y))\in \bS^1$. We claim $|\psi \phi_j|^2 \ind{A_j}\not\equiv 0$. Indeed, otherwise $T(\psi\phi_j)= \psi \phi_j \ind{A_j} \lambda_j=0$ and thus $\psi\phi_j\equiv 0$ which is a contradiction. It follows that it is enough to prove that the measure $\nu$ defined by the Fourier coefficients
$$\hat{\nu}(n):=\int_Y e(t_j(y))^n d\mu_Y  $$
is absolutely continuous with respect to Lebesgue. Let $\eta_{H}:\R^{d-l}\to H^\circ, \text{ and }\eta_{G}: \R^d \to G^\circ$, be Mal'cev coordinates on $H^\circ$ and $G^\circ$ respectively. Since $\chi_j$ is a character in $H$, its restriction to $H^\circ$ must have the form
$$\chi_j|_{H^\circ}(h)= e_j( k_j^T \cdot \eta_H^{-1}(h) ) ,$$
for each $h\in H^\circ$, and for some $k_j\in \Z^{d-l}$. By definition of $t_j$, for $x\in [0,1)^l$ and $g=\eta_G(x)$ we have
$$e(t_j(y))= \chi_j(\tau) \chi_j([g^{-1},\tau] )= \chi_j(\tau) e( k_j^T \cdot \eta_H^{-1}([g^{-1},\tau] ))=\chi_j(\tau) e( k_j^T \cdot \eta_H^{-1}([\eta_G(x)^{-1},\tau] )). $$
Since the operation in $G$ is polynomial in $\R^d$, the function
$$x\in \R^d \to k_j^T \cdot \eta_H^{-1}([\eta_G(x)^{-1},\tau] ) $$
is a polynomial, which we denote by $p_j(x)$. Notice that $\chi_j$ is not $[\tau,G]$-invariant. Indeed, if for the sake of contradiction we assume that $\chi_j$ is $[\tau,G]$-invariant, then we see that $y\mapsto e(t_j(y))$ is constant. This implies that the function $y\mapsto \lambda_j(y)$ has value $0$ or $1$. By our definition of eigenfunction, we have that $\phi_j$ is $T$-invariant, and thus $H$-invariant. Since $[\tau,G]\subseteq H$, we have that $\mathcal{J}([\tau,G],\Gamma)\subseteq H$  as $H$ is a rational connected normal subgroup of $G^\circ$ containing $[\tau,G]$, and $\mathcal{J}([\tau,G],\Gamma)$ is the smallest with such a property. As $\phi_j$ is not $\mathcal{J}([\tau,G],\Gamma)$-invariant, this is a contradiction. Thus, $\chi_j$ is not $[\tau,G]$-invariant, which implies that $p_j$ is not constant. 
We have  
\begin{align*}
  \widehat{\nu}(n)&= \chi(\tau)^n\int_{[0,1)^l} \chi([\eta_G(x)^{-1},\tau] H\cap \Gamma)^n dm_{\R^d}(x)=\chi(\tau)^n \int_{[0,1)^l} e(p_j(x))^n dm_{\R^d}(x).   
\end{align*}
Thus, it is enough to show that the pushforward $p(m_{\R^d})$ is absolutely continuous with respect to the Lebesgue measure of $\R^d$. Since $p$ is not constant, then $\nabla p(x)\neq 0$ for $\mu$-a.e. $x\in \R^d$. By \cref{app-coarea-formula} we have that the push forward of $p_j^*(m_{\R^d})$ is absolutely continuous with respect to the Lebesgue measure, concluding that $ \nu \ll m_{\bS^1}$ and thus $\sigma_f\ll m_{\bS^1} $.

To see that the maximal spectral type is countable Lebesgue, it is enough to work in a factor of the system. Let $s\geq 2$ be the minimum such that the projection of $\mathcal{J}([\tau,G],\Gamma)$ in $G/G_{s}$ is non trivial. Reducing to the factor $G/G_s\Gamma$ of $X$, we can assume without loss of generality that $X=G/\Gamma$ is a $(s-1)$-step nilsystem such that $[\tau,g]\in Z(G)$ for each $g\in G$. Using \cref{PW} with $L=Z(G)$, we obtain the decomposition 
$$L^2(X)=\bigoplus_{\chi\in \widehat{J}}V_\chi, $$
    where $J= L/L\cap \Gamma$ and for $\chi\in \widehat{J}$ with $\chi(L\cap \Gamma)\equiv 1$, $V_\chi$ is defined as in \cref{def-V-chi}. 
    Since $\mathcal{J}([\tau,G],\Gamma)\subseteq Z(G)$ is nontrivial and connected, there are infinitely many $\chi\in \hat{J}$ such that $\chi$ is nontrivial in $\mathcal{J}([\tau,G],\Gamma)$. Take one of such $\chi$ and observe that $\chi([\tau,G])=\bS^1$ by connectedness of $[\tau,G]$ and continuity of $\chi$. Let $g\in G$ and let $f\in V_\chi$ attaining the maximal spectral type $\sigma$ on $V_\chi$. Let $g\in G$ and denote $\alpha:=  \chi([\tau,g])$. Similarly to the proof of \cref{Spectrum-uniform-part}, we have that 
    \begin{equation*}
       \int  \overline{T_gf} \cdot T_\tau^nT_gfd\mu= \int_{\bS^1} z^n d\sigma^\alpha(z),
    \end{equation*}
where we recall that $\sigma^\alpha$ denotes the measure $\sigma$ translated by $\alpha$. Thus, $\sigma^\alpha\ll \sigma$ for all $\alpha\in \chi([\tau,G])=\bS^1    $. By \cref{Lemma-measure-in-circle} we conclude that $\sigma=m_{\bS^1}$, concluding that in each subspace $V_\chi$ the maximal spectral type is the Lebesgue measure $m_{\bS^1}$, finishing the proof.  
\end{proof}

\small{
\bibliographystyle{abbrv}
\bibliography{refs}
}

\bigskip
\noindent
Felipe Hernández\\
\textsc{{\'E}cole Polytechnique F{\'e}d{\'e}rale de Lausanne} (EPFL)\par\nopagebreak
\noindent
\href{mailto:felipe.hernandezcastro@epfl.ch}
{\texttt{felipe.hernandezcastro@epfl.ch}}

\end{document}